\documentclass[11pt]{amsart}

\usepackage{amsmath,amsthm,amssymb,xypic}

\theoremstyle{plain}
    \newtheorem{theorem}{Theorem}[section]
    \newtheorem{lemma}[theorem]{Lemma}
    \newtheorem{corollary}[theorem]{Corollary}
    \newtheorem{proposition}[theorem]{Proposition}

 \theoremstyle{definition}
    \newtheorem{definition}[theorem]{Definition}

    \newtheorem{remark}[theorem]{Remark}

\theoremstyle{remark}

\numberwithin{equation}{section}

 \DeclareMathOperator{\proj}{proj}
  \DeclareMathOperator{\red}{red}

\DeclareMathOperator{\Ad}{Ad}

\DeclareMathOperator{\ind}{index}

\DeclareMathOperator{\grad}{grad}
\DeclareMathOperator{\End}{End}
\DeclareMathOperator{\Hom}{Hom}

\DeclareMathOperator{\Spin}{Spin}
\DeclareMathOperator{\OO}{O}
\DeclareMathOperator{\SO}{SO}

\DeclareMathOperator{\U}{U}

  \DeclareMathOperator{\Res}{Res}

  \DeclareMathOperator{\DInd}{D-Ind}

\DeclareMathOperator{\vol}{vol}
\begin{document}


\newcommand{\myemph}{\emph}

\newcommand{\Spinc}{\Spin^c}

    \newcommand{\R}{\mathbb{R}}
    \newcommand{\C}{\mathbb{C}}
    \newcommand{\N}{\mathbb{N}}
    \newcommand{\Z}{\mathbb{Z}}
    \newcommand{\Q}{\mathbb{Q}}
    \newcommand{\bT}{\mathbb{T}}
    \newcommand{\bP}{\mathbb{P}}

\newcommand{\g}{\mathfrak{g}}
\newcommand{\h}{\mathfrak{h}}
\newcommand{\p}{\mathfrak{p}}
\newcommand{\kg}{\mathfrak{g}}
\newcommand{\kt}{\mathfrak{t}}
\newcommand{\ka}{\mathfrak{a}}
\newcommand{\XX}{\mathfrak{X}}
\newcommand{\kh}{\mathfrak{h}}
\newcommand{\kp}{\mathfrak{p}}
\newcommand{\kk}{\mathfrak{k}}

\newcommand{\cA}{\mathcal{A}}
\newcommand{\cE}{\mathcal{E}}
\newcommand{\calL}{\mathcal{L}}
\newcommand{\calH}{\mathcal{H}}
\newcommand{\cO}{\mathcal{O}}
\newcommand{\cB}{\mathcal{B}}
\newcommand{\cK}{\mathcal{K}}
\newcommand{\cP}{\mathcal{P}}
\newcommand{\cN}{\mathcal{N}}
\newcommand{\calD}{\mathcal{D}}
\newcommand{\cC}{\mathcal{C}}
\newcommand{\calS}{\mathcal{S}}
\newcommand{\cM}{\mathcal{M}}

\newcommand{\cCM}{\cC}
\newcommand{\PM}{P}
\newcommand{\DM}{D}
\newcommand{\LM}{L}
\newcommand{\vM}{v}

\newcommand{\sumGam}{\textstyle{\sum_{\Gamma}} }

\newcommand{\sigDg}{\sigma^D_g}

\newcommand{\Bigwedge}{\textstyle{\bigwedge}}

\newcommand{\ii}{\sqrt{-1}}

\newcommand{\Ubar}{\overline{U}}

\newcommand{\beq}[1]{\begin{equation} \label{#1}}
\newcommand{\eeq}{\end{equation}}

\newcommand{\mattwo}[4]{
\left( \begin{array}{cc}
#1 & #2 \\ #3 & #4
\end{array}
\right)
}

\newenvironment{proofof}[1]
{\noindent \emph{Proof of #1.}}{\hfill $\square$}

\title{An equivariant orbifold index for proper actions}

\author{Peter Hochs} 
\address{University of Adelaide}
\email{peter.hochs@adelaide.edu.au}
\author{Hang Wang}
\address{East China Normal University}
\email{wanghang@math.ecnu.edu.cn}



\date{\today}

\maketitle

\begin{abstract}
For a proper, cocompact action by a locally compact group of the form $H \times G$, with $H$ compact, we
define an $H \times G$-equivariant index of $H$-transversally elliptic operators, which takes values in $KK_*(C^*H, C^*G)$.
This
simultaneously generalises the Baum--Connes analytic assembly map, Atiyah's index of transversally elliptic operators, and Kawasaki's orbifold index. This index also generalises the assembly map to elliptic operators on orbifolds. 
In the special case where the manifold in question is a real semisimple Lie group, $G$ is a cocompact lattice and $H$ is a maximal compact subgroup, we realise the Dirac induction map from the Connes--Kasparov conjecture as a Kasparov product and obtain an index theorem for $\Spin$-Dirac operators on compact locally symmetric spaces.
\end{abstract}

\tableofcontents

\section{Introduction}

Two natural ways in which orbifolds occur are as quotients of locally free actions by compact groups, and of proper actions by discrete groups. In fact, every orbifold can be realised as the quotient of a locally free action by compact group $H$ on a manifold $\tilde M$; see e.g.\ \cite{Farsi92, Vergne96}. A well-known approach to index theory on a compact orbifold $M := \tilde M/H$ is to consider an elliptic operator $D$ on $M$ as a transversally elliptic operator $\tilde D$ on $\tilde M$, and define the index of $D$ as the $H$-invariant part of the $H$-equivariant index of $\tilde D$ in the sense of Atiyah \cite{Atiyah74}.

Intuitively, if an orbifold $M$ is realised as the quotient of a proper action by a discrete group $\Gamma$ on a manifold $X$, then the orbifold index of an operator on $M$ is the $\Gamma$-invariant part of its lift to an operator on $X$. This can be made precise in terms of the Baum--Connes assembly map, with values in the $K$-theory of the maximal group $C^*$-algebra of $\Gamma$, from which the invariant part can be obtained by an application of the map given by summing over $\Gamma$.

Our purpose in this paper is to unify and extend these two approaches to orbifold index theory. Along the way, we construct a generalisation of the Baum--Connes assembly map from manifolds to orbifolds. 
We also realise the Dirac induction map from the Connes--Kasparov conjecture as a Kasparov product.

For a compact group $H$ and a locally compact group $G$, and a proper, isometric, cocompact action by $H \times G$ on a manifold $\tilde M$, we define an index of $H\times G$-equivariant, $H$-transversally elliptic operators on $\tilde M$, with values in $KK_*(C^*H, C^*G)$. This builds on parts of Kasparov $KK$-theoretic treatment of transversally elliptic operators \cite{Kasparov15}.
If $H$ is trivial, then this index is the Baum--Connes assembly map. If $G$ is trivial, it is Atiyah's index of transversally elliptic operators, whose $H$-invariant part is the realisation of the orbifold index on $M = \tilde M/H$ mentioned above, if $H$ acts locally freely.  In general, the  pairing of this index with the class of the trivial representation of $H$ in $K_*(C^*H)$  generalises the Baum--Connes assembly map to $G$-equivariant elliptic operators on the possibly singular space $M$.

Another $K$-theoretic approach to index theory on orbifolds was developed by Farsi \cite{Farsi92}. An index of families of transversally elliptic operators was constructed and applied in a $KK$-theoretic setting by Baldar\'e \cite{Baldare2, Baldare1}.

The index in  $KK_*(C^*H, C^*G)$ is furthest removed form existing index theory if the action by $H$ on $\tilde M$ is not free, and $M$ is not a smooth manifold. However, even if $H$ acts freely, this index refines existing (orbifold) indices, as the double quotient $\tilde M/(H \times G)$ may be singular. We investigate the index in this special case, and find relations with the Baum--Connes assembly map.

A natural setting in which this applies is the case of compact locally symmetric spaces $\Gamma \backslash G /K$, where $G$ is a real semisimple group, $\Gamma < G$ is a cocompact lattice, and $K<G$ is maximal compact. For a class of examples including these spaces, we show that the index of an elliptic operator on $\Gamma \backslash G /K$ can be obtained both as the $K$-invariant part of a transversally elliptic operator on $\Gamma \backslash G$ and as the $\Gamma$-invariant part (in a $K$-theoretic sense) of the index of an elliptic operator on $G/K$. These approaches are unified in the sense that both indices  are obtained from the index of a $K \times \Gamma$-equivariant, transversally elliptic operator on $G$, which generalises and refines the two indices on $\Gamma \backslash G$ and  $G/K$.
In this sense, the index we consider here encodes the most refined index-theoretic information on $\Gamma \backslash G /K$, and simultaneously incorporates the $\Gamma$- and $K$-symmetries.
We obtain explicit expressions for the values of natural traces on the $\Gamma \times K$-equivariant index of the lift of the $\Spin$-Dirac operator from $G/K$ to $G$ in this context.



\subsection*{Acknowledgements}

H. Wang acknowledges support from Thousand Youth Talents Plan and from grants NSFC-11801178 and Shanghai Rising-Star Program 19QA1403200.

\section{Preliminaries and results}

\subsection{Preliminaries}

Let $\tilde M$ be a Riemannian manifold. Let $G$ and $H$ be Lie groups, acting isometrically on $\tilde M$. Suppose that the actions by the two groups commute. Then $G$ has the induced action on the quotient $M := \tilde M/H$. Every orbifold occurs in this way, for a compact Lie group $H$ acting locally freely on $\tilde M$ (see e.g.\ the introductions to \cite{Farsi92, Vergne96}, and Remark \ref{rem On frame}). But we will consider more general spaces of the form $M = \tilde M/H$.

Let $E \to M$ be a Hermitian, $\Z_2$-graded, $G$-equivariant, continuous vector bundle. Let $q\colon \tilde M \to M$ be the quotient map. Suppose that $\tilde E := q^*E \to \tilde M$ has the structure of a smooth vector bundle.

\begin{definition}\label{def ell}
A $G$-equivariant \emph{differential operator} on $E$ is an operator $D$ on $\Gamma^{\infty}(\tilde E)^H$ that is the restriction of a $G\times H$-equivariant differential operator $\tilde D$ on $\Gamma^{\infty}(\tilde E)$. Such a differential operator $D$ is \emph{elliptic} if the operator $\tilde D$ is (or can be chosen to be) transversally elliptic with respect to the action by $H$.
\end{definition}
If $H$ acts properly and freely on $\tilde M$, then this definition reduces to the usual definition of elliptic differential operators.
\begin{remark}
If $D$ is a first order differential operator, and $H$ acts properly and freely, then Definition \ref{def ell} becomes very explicit. In this case $D$ has a unique pullback along any smooth, $G$-equivariant map $f\colon N \to M$ to a linear operator $f^*D$ on $\Gamma^{\infty}(f^*E)$ satisfying
\begin{itemize}
\item for all $s \in \Gamma^{\infty}(E)$,
\[
(f^*D)(f^*s) = f^*(D s);
\]
\item for all $\sigma \in \Gamma^{\infty}(f^*E)$ and $\varphi \in C^{\infty}(N)$,
\[
(f^*D)(\varphi \sigma) = 
 \sigma_D(Tf \circ \grad(\varphi))\sigma + \varphi (f^*D)\sigma. 
\]
\end{itemize}
Here $\sigma_{D}$ is the principal symbol of $D$. Existence and uniqueness of $f^*D$ can be proved in the same way as one proves that pullbacks of connections are well-defined. In this case, we may take $\tilde D = q^*D$ in Definition \ref{def ell}.
\end{remark}

We fix a first order, $G$-equivariant, self-adjoint, elliptic differential operator $D$ on $E$, that is odd with respect to the grading, and a lift $\tilde D$ of $D$ to $\tilde E$ as in Definition \ref{def ell}. Let $\tilde D^+$ and $\tilde D^-$ be the restrictions of $\tilde D$ to sections of the even and odd parts of $\tilde E$, respectively.

If $M$, $G$ and $H$ are compact, then a natural definition of the $G$-equivariant index of $D$ is
\beq{eq ind cpt}
\ind_G(D) = [\ker(\tilde D^+)^H] - [\ker(\tilde D^-)^H]. 
\eeq
This is an element of the representation ring of $G$.
An index formula for such an index was given by Vergne \cite{Vergne96}. It is our goal in this note to generalise this index to noncompact $M$ and $G$, assuming the action to be proper, and $M/G$ to be compact. 


\subsection{The index}

From now on, suppose that $H$ is compact, that $G$ acts properly on $\tilde M$ (and hence on $M$), and that $M/G$ is compact. Then there is a \emph{cutoff function} $c \in C_c(M)$, such that for all $m \in M$,
\[
\int_{G}c(gm)^2\, dg = 1
\]
for a left Haar measure $dg$ on $G$. The pullback $\tilde c := q^*c \in C_c(\tilde M)^H$ then has the analogous property.

Consider the idempotents $p \in C_0(M) \rtimes G$ and $\tilde p \in C_0(\tilde M) \rtimes G$ defined by
\beq{eq p}
\begin{split}
p(m,g)&= c(m)c(g^{-1}m);\\
\tilde p(\tilde m,g) &= \tilde c(\tilde m) \tilde c(g^{-1}\tilde m),
\end{split}
\eeq
for $m \in M$, $\tilde m \in \tilde M$, and $g \in G$. They define classes
\[
\begin{split}
[p]&\in KK(\C, C_0(M)\rtimes G);\\
[\tilde p]&\in KK_H(\C, C_0(\tilde M)\rtimes G).
\end{split}
\]

Let
\[
\pi_{\tilde M, H}\colon C_0(\tilde M)\rtimes H\to \cB(L^2(\tilde E))
\]
be the $*$-representation defined by pointwise multiplication by functions in $C_0(M)$, and the unitary representation of $H$ in $L^2(\tilde E)$. Kasparov showed in Proposition 6.4 in \cite{Kasparov15} that the transversally elliptic operator $\tilde D$ defines a class
\beq{eq tilde D KK}
[\tilde D] := \bigl[L^2(\tilde E), \frac{\tilde D}{\sqrt{\tilde D^2+1}}, \pi_{\tilde M, H} \bigr] \in KK^G_0(C_0(\tilde M)\rtimes H, \C).
\eeq

Let
\[
\begin{split}
j^G\colon&KK_*^G(C_0(\tilde M) \rtimes H, \C)\to KK_*(C_0(\tilde M) \rtimes (G\times H), C^*G);\\
j^H\colon&KK_*^H(\C, C_0(\tilde M)\rtimes G)\to KK_*(C^*H, C_0(\tilde M)\rtimes (G\times H)).
\end{split}
\]
be descent maps, see 3.11 in \cite{Kasparov88}. Here we used the fact that the actions by $G$ and $H$ commute.
(We use either maximal or reduced crossed-products and group $C^*$-algebras.)

Consider the classes
\[
\begin{split}
 j^H[\tilde p] &\in KK_*(C^*H, C_0(\tilde M)\rtimes (G\times H));\\
 j^G[\tilde D]&\in KK_*(C_0(\tilde M) \rtimes (G\times H), C^*G).
\end{split}
\]
Let $[1_H] \in KK(\C, C^*H) = R(H)$ be the class of the trivial representation of $H$.
\begin{definition}\label{def index}
The \emph{$(H,G)$-equivariant index} of $\tilde D$ is
\[
\ind_{H, G}(\tilde D) = j^H[\tilde p] \otimes_{C_0(\tilde M) \rtimes (G\times H)}j^G[\tilde D] \quad \in KK_*(C^*H, C^*G).
\]
The \emph{$G$-equivariant index} of $D$ is
\[
\ind_G(D) = [1_H] \otimes_{C^*H} \ind_{H, G}(\tilde D) \quad \in KK_*(\C, C^*G).
\]

More generally, for any
locally compact group $G$, compact group $H$, and
locally compact, Hausdorff, proper, cocompact $G\times H$ space $X$, the  \emph{$G$-equivariant index map}
\[
\ind_G\colon KK_*^G(C_0(X)\rtimes H, \C) \to K_*(C^*G)
\]
is defined by
\[
\ind_G(a) = [1_H] \otimes_{C^*H}  j^H[\tilde p] \otimes_{C_0(\tilde M) \rtimes (G\times H)}j^G(a),
\]
for $a \in KK_*^G(C_0(X)\rtimes H, \C)$.
\end{definition}
In the definition of $\ind_G(D)$, the Kasparov product with $[1_H]$ plays the role of taking the $H$-invariant part of the $G\times H$-equivariant index of $\tilde D$, analogously to \eqref{eq ind cpt}.

In the rest of this note, we investigate some  properties and applications of these indices. First of all, the $(H, G)$-equivariant index generalises Atiyah's  \cite{Atiyah74} and Kawasaki's \cite{Kawasaki81} classical indices.
\begin{lemma}
If $M$ is compact and $G = \{e\}$ is the trivial group, then $\ind_{H, \{e\}} \tilde D \in KK(C^*H,\C)$ is the index of $\tilde D$ in the sense of Atiyah. In particular, if $H$ acts locally freely on $M$, then $\ind_{\{e\}}(D)$ is Kawasaki's orbifold index of $D$.
\end{lemma}
\begin{proof}
If $M$ is compact, and $G = \{e\}$, then we can take $c$ to be the constant function $1$ on $M$. Then $[\tilde p]$ is the class of the only $*$-homomorphism $\C \to C_0(\tilde M)$. So
\[
\ind_{H\times \{e\}}(\tilde D) = \psi_*[\tilde D] \in KK(C^*H,\C),
\]
where $\psi$ is the map from $\tilde M$ to a point. The right hand side of the above equality is the index of $\tilde D$ in the sense of Atiyah, see Remark 6.7 in \cite{Kasparov15}. Pairing this index with $[1_H]$ means taking its $H$-invariant part, which yields Kawasaki's index of $D$ if $H$ acts locally freely.
\end{proof}

\begin{remark}
Theorem 8.18 in \cite{Kasparov15} is a topological expression for the class \eqref{eq tilde D KK} in terms of the principal symbol of $\tilde D$. This directly implies  analogous $KK$-theoretic index theorems for $\ind_{H,G}(\tilde D)$ and $\ind_{G}(D)$.
\end{remark}

The case where $H$ acts trivially on $\tilde M$ is not the most interesting, but we include it as a consistency check. For any locally compact, Hausdorff, proper, cocompact $G$-space $X$, let
\[
\mu_X^G\colon KK_G(C_0(X),\C)\to KK(\C, C^*G)
\]
be the analytic assembly map \cite{Connes94}.
\begin{lemma}\label{lem H triv}
If $H$ acts trivially on a
locally compact, Hausdorff, proper, cocompact $G$-space $X$,
 then the map
\[
\ind_{G}\colon KK_*^G(C_0(X)\rtimes H, \C) =  
KK_*^G(C_0(X), \C) \otimes \hat R(H)
\to K_*(C^*G)
\]
is given by
\[
\ind_G(a \otimes [V]) = \left\{ 
\begin{array}{ll}
\mu_X^G(a) & \text{if $V=1_H$};\\
0 & \text{otherwise},
\end{array}
\right.
\]
for $a \in KK_*^G(C_0(X), \C)$ and $V \in \hat H$. Here $\mu_X^G$ is the analytic assembly map, and $\hat R(H)$ is the completed representation ring of $H$.
\end{lemma}
\begin{proof}
It follows from the definition of the descent map and the Peter--Weyl theorem that
the descent map
\[
j^H\colon R(H) = KK_0^H(\C, \C) \to KK_0(C^*H, C^*H) = \End_{\Z}(R(H))
\]
maps $[V] \in \hat H$ to the projection map $\proj_V$ onto $\Z[V]$. Now $X/H = X$,  $\tilde c = c$ and $\tilde p = p$, so the class 
\[
[\tilde p] \in KK^H_0(\C, C_0(X)\rtimes G) = KK_0(\C, C_0(X)\rtimes G)  \otimes R(H)
\]
equals $[p] \boxtimes [1_H]$, where $\boxtimes$ is the external Kasparov product. We find that
\[
j^H[\tilde p] = [p] \boxtimes \proj_{1_H} \in KK_0(\C, C_0(X)\rtimes G)  \otimes \End_{\Z} (R(H)).
\]

Now for $a \in KK_*^G(C_0(X), \C)$ and $V \in \hat H$,
\[
j^G(a \otimes [V]) = j^G(a) \boxtimes [V] \in KK_*^G(C_0(X), \C) \otimes \hat R(H).
\]
So
\[
\ind_G(a \otimes V) = [p] \otimes_{C_0(X)\rtimes G} j^G(a) \boxtimes \langle [V], \proj_{1_H}([1_H])\rangle,
\]
which implies the claim. The angular brackets denote the pairing between $\hat R(H) = \Hom_{\Z}(R(H), \Z)$ and $R(H)$.
\end{proof}

\subsection{Free actions by $H$}

If $H$ acts freely on $\tilde M$, then $M$ is a smooth manifold. In that case, the index of Definition \ref{def index} reduces to the analytic assembly map, see Proposition \ref{prop assembly}. In this sense, the $G$-equivariant index generalises the analytic assembly map to orbifolds. If, furthermore, $G = \Gamma$ is discrete, then $M/\Gamma$ is an orbifold. Even if $\Gamma$ acts freely on $\tilde M$, its action on $M$ is not necessarily free. (Similarly, the action by $H$ on $\tilde M /\Gamma$ is not free in general.) This leads to two different ways to realise orbifold indices on spaces of the type that includes compact locally symmetric spaces, Corollary \ref{cor H free}, which is based on  Theorem \ref{thm H free}. We work out the example of locally symmetric spaces in Subection \ref{sec loc symm}.


If $H$ acts freely on $\tilde M$, then $D$ is an elliptic differential operator on the smooth vector bundle $E \to M$ in the usual sense. Then it defines a $K$-homology class
\[
[D] := \bigl[L^2(E), \frac{D}{\sqrt{D^2+1}}, \pi_{M}\bigr]\in KK_G(C_0(M),\C),
\]
where $\pi_M$ is defined by pointwise multiplication.
\begin{proposition}\label{prop assembly}
If $H$ acts freely on $\tilde M$, then
\[
\ind_G(D) = \mu_M^G[D].
\]
\end{proposition}
This proposition will be proved in Section \ref{sec assembly}. 

Consider, for the moment, the case where $G = \Gamma$ is discrete, and $H = \{e\}$ is trivial. Then $M = \tilde M$, $\tilde D = D$ is elliptic, and $\tilde M/\Gamma = M/\Gamma$ is an orbifold. Consider the $*$-homomorphism
\[
\textstyle{\sum_{\Gamma}}\colon C^*_{\max}\Gamma \to \C
\]
given on $l^1(\Gamma) \subset C^*_{\max}\Gamma$ by summing over $\Gamma$. Let $[\sum_{\Gamma}] \in KK(C^*_{\max}\Gamma, \C)$ be the corresponding class.
If we use maximal crossed products and group $C^*$-algebras, then
 Proposition \ref{prop assembly},  and Theorem 2.7 and  Proposition D.3 in \cite{Mathai10}, imply that
\beq{eq sum Gamma}
\ind_{\Gamma}(D)  \otimes_{C^*_{\max}\Gamma} [\textstyle{\sum_{\Gamma}}] =
\bigl(\textstyle{\sum_{\Gamma}}\bigr)_*\ind_{\Gamma}(D) = \dim(\ker(D^+)^{\Gamma}) - \dim(\ker(D^-)^{\Gamma}).
\eeq
The right hand side is the index
of the operator $D_{\Gamma}$ on the compact orbifold $M/\Gamma$ induced by $D$. This gives another realisation of the orbifold index in terms of the index of Definition \ref{def index}.

This construction applies more
 generally. Let $I_G\colon C^*_{\max}(G) \to \C$ be the continuous extension of the integration map on $L^1(G)$, and $[I_G] \in KK(C^*_{\max}G, \C)$ the corresponding $KK$-class.
 \begin{theorem} \label{thm H free}
 Suppose that $G$ is unimodular and that $H$ acts freely on $\tilde M$.
The multiplicity of every irreducible representation of $H$ in $\ker(\tilde D)^G$ is finite, and
\beq{eq H free}
\ind_{H,G}(\tilde D) \otimes_{C^*_{\max}G} [I_G]= [\ker(\tilde D^+)^G] -  [\ker(\tilde D^-)^G]  \quad \in KK(C^*H, \C) = \hat R(H).
\eeq
\end{theorem}

In the setting of Theorem \ref{thm H free}, we denote the right hand side of \eqref{eq H free} by $\ind_H(\tilde D^G)$. If the action by $G$ on $\tilde M$ is free, then this is the index of the transversally elliptic operator $\tilde D^G$ on $\Gamma^{\infty}(\tilde E/G) = \Gamma^{\infty}(\tilde E)^G$ induced by $\tilde D$.
 \begin{corollary} \label{cor H free}
If $G$ is unimodular and $H$ acts freely on $\tilde M$, then $\ker(\tilde D)^{H \times G}$ is finite-dimensional, and
 \beq{eq H G inv}
 \begin{split}
 \dim(\ker(\tilde D^+)^{H \times G}) -  \dim(\ker(\tilde D^-)^{H \times G}) &=[1_H] \otimes_{C^*H}\ind_{H,G}(\tilde D) \otimes_{C^*_{\max}G} [I_G] \\
&= (I_{G})_*(\ind_{G}(D)) \\
&= [1_H] \otimes \ind_{H}(\tilde D^{G}).
 \end{split}
 \eeq
\end{corollary}
If $G = \Gamma$ is discrete, and $H$ acts locally freely, then $\tilde M/(H \times \Gamma)$ is an orbifold, and the left hand side of \eqref{eq H G inv} is the orbifold index of
 the operator on $\tilde M/(H \times \Gamma)$ induced by $\tilde D$.
 Then Corollary \ref{cor H free} shows that $\ind_{H,G}(\tilde D)$ is a common refinement of the two indices $\ind_{G}(D)$ and $\ind_{H}(\tilde D^{G})$,  which refine the orbifold index of the operator $\tilde D^{H \times \Gamma}$ on $\tilde M/(H \times \Gamma)$ induced by $\tilde D$ in two different ways. In this sense, $\ind_{H,G}(\tilde D)$ contains the most refined index-theoretic information about $\tilde D^{H \times \Gamma}$.
 This applies for example in the case of compact locally symmetric spaces, see Section \ref{sec loc symm}.

Theorem \ref{thm H free} and Corollary \ref{cor H free} are deduced from Proposition \ref{prop assembly} in Section \ref{sec H free}.

\begin{remark} \label{rem On frame}
Let $M$ be a complete Riemannian manifold of dimension $n$ and let $\Gamma$ be a discrete group acting properly, cocompactly and isometrically on $M$.
Then $X := M/\Gamma$ is a compact orbifold.
Let $P$ be the $\OO(n)$-frame bundle of $X$.
Denote $H=\OO(n)$. Then
$P$ is a compact manifold acted on freely by $H$.
One can lift the $H$-frame bundle $P$ from $X$ to $M$ to obtain a principal $H$-bundle $\tilde M$ over $M$.
Then $\tilde M$ has free, commuting actions by $H$ and $\Gamma$, and
\[
X=P/H = \tilde M/(\Gamma \times H).
\]

Let $D_X$ be an elliptic differential operator on $X$. It can be realised as either a $\Gamma$-equivariant elliptic differential operator $D_M^{\Gamma}$ on $M$, or an $H$-transversally elliptic operator $D_P^H$ on $P$. These two operators have a common lift to an $H \times \Gamma$-equivariant, $H$-transversally elliptic operator $\tilde D$ on $\tilde M$. Corollary \ref{cor H free} implies that the orbifold index of $D_X$ can be obtained from the $(H,\Gamma)$-index of $\tilde D$ as
\begin{multline*}\label{eq ind DX}
\ind(D_X) = [1_H] \otimes_{C^*H} \ind_{H,\Gamma}(\tilde D) \otimes_{C^*_{\max}\Gamma} [\sumGam]\\
= (\sumGam)_*(\ind_{\Gamma}(D_M^{\Gamma}))
= [1_H] \otimes_{C^*H} \ind_H(D_P^H).
\end{multline*}

It is an interesting question in what way the contributions to $\ind(D_X)$ from singularities in the quotient of $M$ by $\Gamma$ or the quotient of $P$ by $H$ are related.

\end{remark}

\section{Proofs of the results}

\subsection{The analytic assembly map}\label{sec assembly}

Let us prove Proposition \ref{prop assembly}. Suppose that $H$ acts freely on $\tilde M$, so that $M$ is smooth. Then we have a Morita equivalence bimodule $\cM$ between $C_0(\tilde M) \rtimes H$ and $C_0(M)$, see Situation 2 in \cite{Rieffel82}. This is a left $C_0(\tilde M)\rtimes H$ and right Hilbert $C_0(M)$-bimodule, defined as the completion of $C_c(\tilde M)$ in the inner product
\[
(\varphi_1, \varphi_2)_{C_0(M)}(m) = \int_{H}\bar \varphi_1(hm)\varphi_2(hm)\, dh,
\]
for $\varphi_1, \varphi_2 \in C_c(\tilde M)$ and $m \in M$. The right action by $C_0(M)$ on $\cM$ is defined by
pointwise multiplication after pullback along $q$. The left action by $C_0(\tilde M)\rtimes H$, denoted by $\pi_{\cM}$, is defined by the standard actions by $C_0(\tilde M)$ and $H$ on $C_c(\tilde M)$. This yields an invertible class
\[
[\cM] := [\cM, 0, \pi_{\cM}] \in KK_G(C_0(\tilde M) \rtimes H, C_0(M)).
\]
\begin{lemma} \label{lem ME}
We have
\[
[\tilde D] = [\cM] \otimes_{C_0(M)} [D] \quad \in KK_G(C_0(\tilde M) \rtimes H, \C).
\]
\end{lemma}
\begin{proof}
We use the unbounded picture of $KK$-theory \cite{Baaj83, Kucerovsky97, Mesland14}. Denoting sets of unbounded $KK$-cycles by the letter $\Psi$, we have
\[
\begin{split}
(\tilde D) = (L^2(\tilde E), \tilde D, \pi_{\cM}) &\in \Psi_G(C_0(\tilde M) \rtimes H,\C);\\
(D) = (L^2(E), D, \pi_{M}) &\in \Psi_G(C_0(M),\C);\\
(\cM) = (\cM, 0, \pi_{M,H}) &\in \Psi_G(C_0(\tilde M) \rtimes H, C_0(M)).
\end{split}
\]
We will show that
\beq{eq kas prod}
(\tilde D) = (\cM) \otimes_{C_0(M)} (D).
\eeq

First note that we have an isomorphism of $C_0(\tilde M) \rtimes H$-modules
\[
\cM \otimes_{C_0(M)}L^2(E) \xrightarrow{\cong}L^2(\tilde E),
\]
mapping $\varphi \otimes s$ to $\varphi q^*s$, for $\varphi \in C_c(\tilde M)$ and $s \in L^2(E)$.
Now Theorem 13 in \cite{Kucerovsky97} states that the equality \eqref{eq kas prod} holds if
\begin{enumerate}
\item for all $\varphi \in C^{\infty}_c(\tilde M)$, the operators
\[
\begin{split}
\tilde D \circ T_{\varphi}-T_{\varphi}\circ D\colon & \Gamma^{\infty}_c(E) \to L^2(\tilde E) \quad \text{and}\\
D \circ T_{\varphi}^*-T_{\varphi}^*\circ \tilde D\colon & \Gamma^{\infty}_c(\tilde E) \to L^2(E)
\end{split}
\]
are bounded, where $T_{\varphi}$ denotes tensoring with $\varphi$;
\item the resolvent of $\tilde D$ is compatible with the zero operator in the sense of Lemma 10 in \cite{Kucerovsky97}, which is a vacuous condition; and
\item a positivity condition that trivially holds because the operator in the cycle $(\cM)$ is zero.
\end{enumerate}
To verify the first condition, we note that, since $\tilde D$ is a first order operator,
\[
\tilde D \circ T_{\varphi}-T_{\varphi}\circ D = \sigma_{\tilde D}(d\varphi) \otimes 1,
\]
which is a bounded operator. And $D \circ T_{\varphi}^*-T_{\varphi}^*\circ \tilde D$ is minus the adjoint of the above operator, hence also bounded.
\end{proof}

Next, consider the maps
\begin{multline} \label{eq composition}
[\tilde p] \in  KK_H(\C, C_0(\tilde M)\rtimes G)\xrightarrow{j^H}
 KK(C^*H, C_0(\tilde M)\rtimes (G\times H))\\
 \xrightarrow{\relbar \otimes_{C_0(\tilde M) \rtimes(G\times H)} j^G[\cM]}
 KK(C^*H, C_0(M)\rtimes G)\xrightarrow{[1_H] \otimes_{C^*H} \relbar}
 KK(\C, C_0(M)\rtimes G) \ni [p].
\end{multline}
\begin{lemma} \label{lem p}
The composition of the maps \eqref{eq composition} maps the class $[\tilde p]$ to $[p]$.
\end{lemma}
\begin{proof}
We have
\[
[\tilde p] = [\tilde p(C_0(\tilde M)\rtimes G), 0, \pi_{\C}],
\]
where $\pi_{\C}$ is the  representation of $\C$ by scalar multiplication. Hence
\[
j^H[\tilde p] =[\tilde p_H (C_0(\tilde M)\rtimes (G\times H)), 0, \pi_{H}],
\]
where $\pi_H$ is the representation of $C^*H$ defined by convolution on $H$, and
 $\tilde p_H$ is the idempotent in $C_0(\tilde M)\rtimes (G\times H)$ defined by extending $\tilde p$ constant in the $H$-direction. (Recall that $H$ is compact.)

Now let $\cM_G$ be the $(C_0(\tilde M) \rtimes(G\times H), C_0(M)\rtimes G)$-bimodule constructed from $\cM$ as in the definition of $j^G$. Then $\cM_G$ implements the Morita equivalence between $C_0(\tilde M) \rtimes(G\times H)$ and $C_0(M)\rtimes G$, and 
\[
\tilde p_H (C_0(\tilde M)\rtimes (G\times H)) \otimes_{C_0(\tilde M) \rtimes(G\times H)}\cM_G = p(C_0(M)\rtimes G).
\]
So
\[
j^H[\tilde p] \otimes_{C_0(\tilde M) \rtimes(G\times H)} j^G[\cM] = [p(C_0(M)\rtimes G), 0, \pi_H].
\]
Finally, pairing with $[1_H]$ means replacing $\pi_H$ by $\pi_{\C}$, so the claim follows.
\end{proof}

\begin{proofof}{Proposition \ref{prop assembly}}
By one of the equivalent definitions of the analytic assembly map, and by Lemma \ref{lem p}, we have
\[
\begin{split}
\mu_M^G[D] &= [p] \otimes_{C_0(M)\rtimes G}j^G[D]\\
&=[1_H] \otimes_{C^*H} j^H[\tilde p] \otimes_{C_0(\tilde M) \rtimes(G\times H)} j^G[\cM]
 \otimes_{C_0(M)\rtimes G}j^G[D] \\
 &=[1_H] \otimes_{C^*H} j^H[\tilde p] \otimes_{C_0(\tilde M) \rtimes(G\times H)} j^G([\cM]
 \otimes_{C_0(M)}[D]).
\end{split}
\]
So the claim follows from Lemma \ref{lem ME}.
\end{proofof}

\subsection{Proofs of Theorem \ref{thm H free} and Corollary \ref{cor H free}} \label{sec H free}

We now deduce Theorem \ref{thm H free} from Proposition \ref{prop assembly} and the results in the appendix to \cite{Mathai10}, and then deduce Corollary \ref{cor H free} from Theorem \ref{thm H free}. We still assume that $H$ acts freely on $\tilde M$.

Let $V \in \hat H$. We will write
\beq{eq V1 V2}
\begin{split}
[V]_1 &\in KK(C^*H, \C);\\
[V]_2 &\in KK(\C, C^*H)
\end{split}
\eeq
for the classes defined by $V$.
Consider the elliptic operator $\tilde D \otimes 1_V$ on $\Gamma^{\infty}(\tilde E) \otimes V$.
Let $E_V \to M$ be the quotient of $\tilde E \otimes V$ by $H$. (Note that $M$ is smooth and $E_V$ is a well-defined vector bundle because $H$ acts freely on $\tilde M$.)
Let $D_V$ be the restriction of $\tilde D \otimes 1_V$ to
\[
\Gamma^{\infty}(M, E_V) = (\Gamma^{\infty}(\tilde E) \otimes V)^H.
\]
This is a $G$-equivariant, elliptic operator. The following result generalises Proposition \ref{prop assembly}, which we use in its proof.
\begin{proposition} \label{prop DV}
For all $V \in \hat H$,
\[
\mu_M^G[D_V] = [V]_2 \otimes_{C^*H} (\ind_{H,G}(\tilde D)).
\]
\end{proposition}
\begin{proof}
For any  $G$-$C^*$-algebra  $B$ and  $(G \times H \times H)$-$C^*$-algebra $A$, let
\[
\Res^{H \times H}_{\Delta(H)}\colon KK^G(A \rtimes(H \times H) ,B ) \to K^G(A \rtimes H, B)
\]
be defined by restriction to the diagonal in $H \times H$ (and similarly for the case where $G = \{e\}$).
Consider the action by $G \times H \times H$ on $\tilde M$, where the second factor $H$ acts trivially.
Consider the diagram
\beq{eq diag H free}
\xymatrix{
KK_0^G(C_0(\tilde M) \rtimes(H \times H), \C) \ar[rr]^-{\ind_{H \times H, G}}\ar[d]_-{\Res^{H \times H}_{\Delta(H)}}&& KK_0(C^*(H \times H), C^*G) \ar[d]_-{\Res^{H \times H}_{\Delta(H)}}\\
KK_0^G(C_0(\tilde M) \rtimes H, \C) \ar[rr]^-{\ind_{H, G}} && KK_0(C^*H, C^*G) \ar[d]_-{[1_H] \otimes_{C^*H} \relbar}\\
KK_0^G(C_0( M) , \C) \ar[rr]^-{\mu_M^G} \ar[u]^-{[\cM] \otimes_{C_0(M)} \relbar}_-{\cong}&& KK_0(\C, C^*G)
}
\eeq
The bottom part of this diagram commutes by Proposition \ref{prop assembly}. Because the second factor $H$ acts trivially on $\tilde M$, the projection $\tilde p$ also defines a class in $KK_{H\times H}(\C, C_0(\tilde M) \rtimes G)$, which we denote by $[\tilde p]_{H \times H}$. Let
\[
1_{C^*H} \in KK(C^*H, C^*H)
\]
be the identity element. Again, because
 the second factor $H$ acts trivially on $\tilde M$, we have
\begin{multline} \label{eq jHp}
j^{H \times H}[\tilde p]_{H \times H} =  j^H[\tilde p] \boxtimes 1_{C^*H}\\
\in KK(C^*(H \times H), C_0(\tilde M) \rtimes (G \times H \times H)) = KK(C^*H \otimes C^*H, C_0(\tilde M) \rtimes (G \times H) \otimes C^*H),
\end{multline}
where $\boxtimes$ denotes the exterior Kasparov product. This equality implies that the top part of \eqref{eq diag H free} commutes. By Lemma \ref{lem ME}, we have
\[
[\cM] \otimes_{C_0(M)} [D_V] = \Res^{H \times H}_{\Delta(H)}[\tilde D \otimes 1_V].
\]
Here we view $\tilde D \otimes 1_V$ as a $G \times H \times H$-equivariant operator, where the second factor $H$ acts trivially on $\tilde M$, and on $E \otimes V$ via its action on $V$. By commutativity of \eqref{eq diag H free}, we therefore have
\beq{eq H free 1}
\mu_M^G[D_V] = [1_H] \otimes_{C^*H} \bigl(\Res^{H \times H}_{\Delta(H)} (\ind_{H\times H, G}(\tilde D \otimes 1_V))\bigr).
\eeq

Next, again using the fact that the second factor $H$ acts trivially on $\tilde M$, we have
\[
[\tilde D \otimes 1_V] = [\tilde D] \boxtimes [V]_1 \quad \in KK_G(C_0(\tilde M) \rtimes(H \times H), \C) = KK_G((C_0(\tilde M) \rtimes H) \otimes C^*H, \C).
\]
Here  $\boxtimes$ again denotes the exterior Kasparov product. By this equality and \eqref{eq jHp},
\beq{eq H free 2}
\begin{split}
\ind_{H\times H, G}(\tilde D \otimes 1_V) &=
(j^{H}[\tilde p] \boxtimes 1_{C^*H}) \otimes_{C_0(\tilde M) \rtimes (G \times H) \otimes C^*H} (j^G[\tilde D] \boxtimes [V]_1)\\
&= \ind_{H, G}(\tilde D) \boxtimes [V]_1.
\end{split}
\eeq

Finally, we have for all $C^*$-algebras $A$ and all $x \in KK(C^*H, A)$,
\[
[1_H] \otimes_{C^*H} \Res^{H \times H}_{\Delta(H)} (x \boxtimes [V]_1) = [V]_2 \otimes_{C^*H} x.
\]
Combining this equality with \eqref{eq H free 1} and \eqref{eq H free 2}, we conclude that the desired equality holds.
\end{proof}

\begin{proofof}{Theorem \ref{thm H free}}
Proposition \ref{prop DV} implies that
\[
\begin{split}
(I_G)_*\ind_{H,G}(\tilde D) &= \bigoplus_{V \in \hat H} ([V]_2 \otimes_{C^*H}(I_G)_*\ind_{H, G} (\tilde D)) \otimes [V]_1\\
&=  \bigoplus_{V \in \hat H} ((I_G)_*\mu_M^G (D_V)) \otimes [V]_1
\end{split}
\]
By unimodularity of $G$,
Theorem 2.7  and  Proposition D.3 in \cite{Mathai10}  imply that the latter expression equals
 \[
 \bigoplus_{V \in \hat H} \bigl(\dim(\ker (D_V^+)^G) - \dim(\ker (D_V^-)^G) \bigr) \otimes [V]_1 = [\ker(\tilde D^+)^G] - [\ker(\tilde D^-)^G].
 \]
\end{proofof}

\begin{proofof}{Corollary \ref{cor H free}}
Associativity of the Kasparov product 
and Theorem \ref{thm H free} imply that
\[
[1_H] \otimes_{C^*H} \ind_H(D^{G}) = (\textstyle{I_{G}})_* \ind_{G}(D).
\]
Because $G$ is unimodular, Theorem 2.7  and  Proposition D.3 in \cite{Mathai10} imply that
the right hand side equals
\[
\dim(\ker(\tilde D^+)^{H \times G}) - \dim(\ker(\tilde D^-)^{H \times G}).
\]
\end{proofof}

\section{Symmetric spaces and locally symmetric spaces}

In this section, we consider a Lie group $G$, a maximal compact subgroup $K<G$,  
the manifold $\tilde M = G$, and the operator $\tilde D$ which is the pullback of the $\Spin$-Dirac operator on $G/K$. Then we obtain a realisation of Dirac induction as a Kasparov product, and an index formula for compact locally symmetric spaces.

\subsection{Dirac induction}\label{sec DInd}

Using the $(H,G)$-equivariant index from Definition \ref{def index}, we can realise the Dirac induction map from the Connes--Kasparov conjecture \cite{Connes94, CEN, Lafforgue02b, Wassermann87} as a Kasparov product.

We will sometimes consider a particular transversally elliptic Dirac-type operator. 
Let $G$ be an almost connected Lie group, and let $K<G$ be maximal compact. Let $\kp \subset \kg$ be the orthogonal complement of $\kk$ with respect to an $\Ad(K)$-invariant inner product on $\kg$. Suppose that the adjoint representation
$
K \to \SO(\kp)
$
of $K$ in $\kp$ lifts to a homomorphism
\beq{eq tilde Ad}
K \to \Spin(\kp).
\eeq
(This is true of we replace $G$ by a double cover if necessary.) Let $\Delta_{\kp}$ be the standard representation of $\Spin(\kp)$. We view it as a representation of $K$ via the map \eqref{eq tilde Ad}.

Let $\{X_1, \ldots, X_l\}$ be an orthonormal basis of $\kp$.
We denote the left regular representation of $G$ in $C^{\infty}(G)$ by $L$ and the Clifford action by $\kp$ on $\Delta_{\kp}$ by $c$.
For $V \in \hat K$, consider the operator
\beq{eq DGKV}
D_{G,K}^V := \sum_{j=1}^{l} L(X_j) \otimes c(X_j) \otimes 1_V
\eeq
on the space
\beq{eq secs GV}
C^{\infty}(G) \otimes \Delta_{\kp} \otimes V
\eeq
of sections of the trivial $G\times K$-equivariant vector bundle
\[
G \times (\Delta_{\kp} \otimes V) \to G.
\]

The operator $D_{G,K}^V$ is  $G\times K$-equivariant, and $K$-transversally elliptic. So we in particular have the element
\beq{eq DInd}
\DInd :=
\ind_{K,G} (D_{G,K}^{\C}) \in KK_*(C^*K, C^*G).
\eeq
Here $\C$ is the trivial representation of $K$.
If $G/K$ is even-dimensional, then $\Delta_{\kp}$ is $\Z_2$-graded, and this element lies in even $KK$-theory. For odd-dimensional $G/K$, $\Delta_{\kp}$ is ungraded, and this index lies in odd $KK$-theory.

\begin{proposition} \label{prop DInd}
For all $V \in \hat K$,
\[
[V]_2 \otimes_{C^*K} \DInd = \DInd_K^G[V] \quad \in K_*(C^*G),
\]
where $[V]_2$ is as in \eqref{eq V1 V2}, and on the right hand side, $\DInd_K^G$ is the Dirac induction map.
\end{proposition}
\begin{proof}
Let $V \in \hat K$. Let $D_{G/K}^V$ be the restriction of $D_{G,K}^V$ to the space of $K$-invariant elements of \eqref{eq secs GV}, which is the space of sections of the vector bundle
\[
G \times_K (\Delta_{\kp} \otimes V) \to G/K.
\]
Proposition
\ref{prop DV} implies that for all $V \in \hat K$,
\[
[V] \otimes_{C^*K} \DInd = \mu_{G/K}^G[D_{G/K}^V] = \DInd_K^G[V].
\]
\end{proof}

Let $\eta \in KK_*(C^*G , C^*K)$ be the dual-Dirac element; see for example Section 2.2 of \cite{Lafforgue02}. By Proposition \ref{prop DInd}, a sufficient condition for injectivity of Dirac induction is \[
\DInd \otimes_{C^*G} \eta = 1_{C^*K} \quad \in KK_0(C^*K, C^*K),
\]
whereas a sufficient condition for surjectivity is
\[
\eta \otimes_{C^*K} \DInd  = 1_{C^*G} \quad \in KK_0(C^*G, C^*G).
\]
Bijectivity of Dirac induction was proved in \cite{CEN, Lafforgue02b}. See also forthcoming work by Higson, Song and Tang.

\subsection{An index theorem for $\Spin$-Dirac operators on compact locally symmetric spaces} \label{sec loc symm}
%
%

Consider
a compact locally symmetric space: an orbifold of the form $X=\Gamma\backslash G/K$ where $G$ is a connected, semisimple Lie group, $K<G$ is a maximal compact subgroup and $\Gamma$ is a cocompact, discrete subgroup in $G$.

%

In several contexts \cite{Gong, Samurkas,Wangwang}, it was shown that traces defined by
 \emph{orbital integrals} on discrete groups are useful tools to extract information from classes in $K$-theory of group $C^*$-algebras. (This is also true for orbital integrals on semisimple Lie groups \cite{HochsWang19} and higher analogues \cite{ST19}.) For a discrete group $\Gamma$, the orbital integral of a function $f \in l^1(\Gamma)$ over a conjugacy class $(\gamma)$ of an element $\gamma \in \Gamma$, is the sum of $f$ over $(\gamma)$:
\[
\tau_{\gamma}(f) = \sum_{h \in (\gamma)}f(h).
\]
We assume that this trace $\tau_{\gamma}$ extends to a continuous functional on a dense subalgebra $\cA(\Gamma) \subset C^*\Gamma$, closed under holomorphic functional calculus (a smooth subalgebra for short). For example, this is  true for every group if $\gamma = e$, and for every $\gamma$ if $G$ has real rank one.
\begin{lemma} \label{lem taug ext}
If $G$ has real rank one, then $\tau_{\gamma}$ defines a continuous linear functional on a smooth subalgebra of $C^*_{\red}\Gamma$.
\end{lemma}
\begin{proof}
If $G$ has real rank one, then $G/K$ has negative sectional curvature, and hence it is a hyperbolic space in the sense of  Definition 3.1 in \cite{Puschnigg10}. Since $\Gamma$ acts cocompactly on $G/K$, the Svar\v{c}--Milnor lemma implies that it is quasi-isometric to $G/K$, with respect to any word-length metric. Hyperbolicity of metric spaces is preserved by quasi-isometries (see 7.2 in \cite{Gromov87}), so that $\Gamma$ is hyperbolic. See also Section 2.7 in \cite{Gromov87}.
Proposition 5.5 in \cite{Puschnigg10} implies that an algebra as in the lemma exists if $\Gamma$ is hyperbolic.
%
\end{proof}
Examples of groups with real rank one are $\OO(n,1)$ and $\U(n,1)$. If the real rank of $G$ is at least two, then $\Gamma$ is not hyperbolic in general. (Then the sectional curvature of $G/K$ does not have a negative upper bound, as $G/K$ admits an embedding of a Euclidean space of dimension at least two.)

By Corollary \ref{cor H free}, in a sense the most refined index-theoretic information about an eliptic operator $D_X$ on the compact symmetric space $X = \Gamma \backslash G/K$ is the $(K,\Gamma)$-index of its lift $\tilde D$ to $G$.
%
A natural way to obtain potentially computable numbers from the index of $\tilde D$ in $KK(C^*K, C^*\Gamma)$ is to evaluate the component in $\hat R(K)$ at a group element (where this makes sense), and to apply suitable traces to the component in $C^*\Gamma$, such as traces defined by orbital integrals.

Because of our assumption that $\tau_{\gamma}$ extends to a smooth subalgebra $\cA(G)$ of $C^*G$, it
 induces a map
\[
\tau_{\gamma}\colon K_0(C^*\Gamma)  = K_0(\cA(G))\to \C.
\]
Therefore, it is a natural problem to compute the element
\beq{eq taugamma index}
\tau_{\gamma}(\ind_{K, \Gamma}(\tilde D)) \in \hat R(K) \otimes \C.
\eeq
We will do this for the operator $D_{G, K}^{\C}$ as in \eqref{eq DGKV}.
We denote the character of a
 finite-dimensional representation space $V$ of $K$, 
 by $\chi_V$. If $V$ is $\Z_2$ -graded, as $V = \Delta_{\kp}$ is, then $\chi_V$ is the difference of the characters of the even and odd parts of $V$.
\begin{proposition} \label{prop indgamma}
%
%

\begin{enumerate}
\item[(a)] if $\gamma$ is not a torsion element and nontrivial, then $\tau_{\gamma}(\ind_{K, \Gamma}(\tilde D)) = 0$.
\end{enumerate}
Suppose that $G$ has discrete series representations, and let $D_{G,K}^{\C}$ be as in \eqref{eq DGKV}.
\begin{enumerate}
\item[(b)] 
For $\gamma = e$,
\[
\tau_{e}(\ind_{K,\Gamma}(D_{G,K}^{\C}))=\vol(\Gamma\backslash G)
\sum_{\pi}
d_{\pi}
[V_{\pi}],
\]
where $d_{\pi}$ is the formal degree of the discrete series representation $\pi$ of $G$, and $V_{\pi} \in \hat K$ is the irreducible representation corresponding to $\pi$ via Dirac induction.
\item[(c)] If $\gamma$ is a regular element of a compact Cartan subgroup of $G$ contained in $K$, then 
\[
\tau_{\gamma}(\ind_{K, \Gamma}(D_{G,K}^{\C})) =
 (-1)^{\dim(G/K)/2}
  \frac{\vol(Z_{G}(\gamma)/Z_{\Gamma}(\gamma))}{\chi_{\Delta_{\kp}}(\gamma)}
 \sum_{[V] \in \hat K} \chi_V(\gamma)[V].
\]
\end{enumerate}
%
%
%
\end{proposition}
\begin{proof}
For $V \in \hat K$, let $D_V$ be as in Subsection \ref{sec H free}. 
By Proposition \ref{prop DV}, 
\begin{multline} \label{eq ind K Gamma D}
\ind_{K, \Gamma}(D_{G,K}^{\C}) = \sum_{V \in \hat K} [V] \otimes \mu_{G/K}^{\Gamma}(D_V) \\
 \in \hat R(K) \otimes K_*(C^*\Gamma) = KK(C^*K, C^*\Gamma).
\end{multline}
%

If $\gamma$ is not a torsion element, then it has no fixed points in $G/K$. As a special case of Theorem 5.10 in \cite{Wangwang}, this implies that $\tau_{\gamma}(\ind_{\Gamma}(D_V)) = 0$, so part (a) follows.



For a semisimple element $g \in G$, let $\tau^G_g$ be the corresponding orbital integral:
\[
\tau_g^G(f) = \int_{G/Z_G(g)}f(xgx^{-1})\, d(xZ_G(g)),
\]
for $f$ such that this converges. 

If 
$\tilde D = D_{G,K}^{\C}$ as in \eqref{eq DGKV}, then for all discrete series representations $\pi$ of $G$,
%
%
the index formula Theorem 6.12 in \cite{Wang14} implies in particular that 
\[
\tau_e (\mu_{G/K}^{\Gamma}(D_{V_{\pi}}))= \vol(\Gamma \backslash G) \tau^G_e(\ind_G(D_{V_{\pi}}))
\]
By 
 Lemma 5.4 in \cite{HW18}, $\tau^G_e (\ind_G(D_{V_{\pi}}) )= d_{\pi}$. 
And if $V \in \hat K$ does not correspond to a discrete series representation, then $\ind_{e}D_{V_{\pi}} = 0$. See Corollary 7.3.B in \cite{Connes82}. So part (b) follows.
%
%
%
%
%

For part (c), we use the fact that 
\[
\tau_{\gamma}(\mu_{G/K}^{\Gamma}(D_V)) = \vol(Z_{G}(\gamma)/Z_{\Gamma}(\gamma))
\tau_{\gamma}^G(\mu_{G/K}^{G}(D_V)).
\]
This follows from the topological expression for these indices in Theorem 5.10 in \cite{Wangwang}. So
\[
\tau_{\gamma}(\ind_{K, \Gamma}(\tilde D)) =
\vol(Z_{G}(\gamma)/Z_{\Gamma}(\gamma))
 \sum_{[V] \in \hat K} \tau_{\gamma}^G(\mu_{G/K}^{G}(D_V)) [V].
 \]
Part (c) now follows from Theorem 3.2 in \cite{HochsWang19}. 
\end{proof}
A version of part (b) of  Proposition \ref{prop indgamma}
was used by  Atiyah and Schmid to obtain formal degrees of discrete series of $G$ \cite{Atiyah77}.

\begin{remark}
For general $\Gamma \times K$-equivariant, $K$-transversally elliptic operators $\tilde D$, a topological expression for the coefficients of all irreducible representations of $V$ in \eqref{eq taugamma index} can be found by combining \eqref{eq ind K Gamma D} with a generalisation of Theorem 6.1 in \cite{Wangwang} to arbitrary elliptic operators, as in the proof of Theorem 2.5 in \cite{HW18}.
\end{remark}

\bibliographystyle{plain}
\bibliography{mybib}

\end{document}